\documentclass[fleqn]{amsart}

\usepackage{latexsym, amsfonts, amsmath, amssymb, euscript}
\newtheorem{thm}{Theorem}[section]
\newtheorem{defn}[thm]{Definition}
\newtheorem{prop}[thm]{Proposition}
\newtheorem{cor}[thm]{Corollary}
\newtheorem{lem}[thm]{Lemma}
\newtheorem{rem}[thm]{Remark}
\newtheorem{exa}[thm]{Example}

\newcommand{\norm}[3]{\ensuremath{\left\Vert#1\right\Vert_{#2}^{#3}}}

\newcommand{\eps}{\varepsilon}
\begin{document}

\title[]{Admissibility of Frech\'et spaces}
\author[]{ Maciej Ciesielski}
\address{Institute of Mathematics, Poznan University of Technology, Piotrowo 3A, 60-965 Poznan, Poland}
\email{maciej.ciesielski@put.poznan.pl}
 \author[]{ Grzegorz Lewicki}
 \address{Department of Mathematics and Computer Science,
 Jagiellonian University,
30-048 Krak\'ow, \L ojasiewicza 6,  Poland }
\email{grzegorz.lewicki@im.uj.edu.pl}

\date{\today}
\begin{abstract}
The aim of this paper is to to show the admissibility of some  class of Frech\'et spaces (see Definition \ref{def32}). In particular, this generalizes the main results of \cite{CL}.
As an application, we show the admissibility of a large class modular spaces equipped with $F$-norms determined in Theorem \ref{increasing}. It is worth noticing that 
$F$-norms introduced in Theorem \ref{increasing} generalize the classical Luxemburg $F$-norm. Also a linear version of admissibility (so called metric approximation property) for order continuous symmetric spaces will be demonstrated  (see Theorem \ref{thm:approx:1}).

\end{abstract}
\maketitle


\noindent{\em AMS~Classification:} 46A80, 46E30. 
\section{Introduction} 
The notion of admissibility, introduced by Klee in \cite{Klee}, 
allows
one to approximate the identity on compact sets by finite-dimensional maps.
Locally convex
spaces are admissible  (see \cite{Na}), and a large literature is devoted to prove that particular classes of non-locally convex function spaces are admissible \cite{CLTT,R1,R2,M,I}. It is important to notice that not all non-locally convex spaces are admissible, in \cite{Cauty} Cauty provides an
example of a metric linear space in which the admissibility fails.

The aim of this paper is to prove the admissibility of a large class of Frech\'et spaces introduced in Definition \ref{def32} (so called {\it $F$-admisssible spaces}; see Theorem \ref{thm1}.) In particular, this generalizes earlier results obtained 
in \cite{CL} for modular function spaces introduced by W. M. Koz\l owski in \cite{Koz}. As an application we prove two fixed point theorems in $F$- admissible spaces (see Theorem \ref{fixpoint} and Theorem \ref{fixpoint2}). Also a linear version of admissibility (so called metric approximation property) for order continuous symmetric spaces will be demonstrated  (see Theorem \ref{thm:approx:1}). Next we apply these results to the large class of modular spaces equipped with $F$-norms introduced in Theorem \ref{increasing}. 
The main interest of the admissibility of modular spaces lies in the possibility of applying the result to the fixed point theory. The fixed point theory in modular spaces, initiated in 1990 by M. A. Khamsi, W. M.
Koz\l owski and S. Reich \cite{KKR},
is quite a recent topic in the theory of nonlinear operators, see e.g. \cite{Kh1}, \cite{Kh2}, \cite{KhKoz1}, \cite{KhKoz2}. The advantage of the theory is that even though a metric may not be defined, many problems in metric fixed point theory can still be possibly formulated in modular spaces.

Now we recall the definition of admissibility.
\begin{defn} {\rm \cite{Klee}
 Let $E$ be a Haudorff topological vector space. A subset $Z$ of
$E$
 is said to be} admissible {\rm if for every compact subset $K$ of $Z$
and for every neighborhood $V$ of zero in $E$ there exists a
continuous mapping $H:  K \to Z$ such that dim(span $[H(K)]$)$<
+\infty$ and $x-Hx \in V$ for every $x \in K$. If $Z=E$ we say that
the space  $E$ is} admissible.
\end{defn}
\section{Preliminary results}
We start with the following defintions.
\begin{defn}
\label{def31a}
Let $X$ be a linear space over $\mathbb{K}$, $\mathbb{K} =\mathbb{R}$ or $\mathbb{K} = \mathbb{C}.$ 
A mapping $ |\cdot |_F:X\rightarrow[0,+\infty)$ is said to be an $F$-norm if it satisfies the following conditions
\begin{itemize}
	\item[$(i)$] $|x|_F = 0$ if and only if $x=0$;
	\item[$(ii)$] $ |x|_F = |ax|_F $ for all $x\in X,$ $a \in \mathbb{K}, $ $|a|=1$; 
	\item[$(iii)$] $|x+y|_F \leq |x|_F + |y|_F $ for all $x,y\in{X}$;
	\item[$(iv)$] $|\lambda_n x_n-\lambda x|_F \rightarrow{0}$ whenever $|x_n-x|_F\rightarrow{0}$ and $\lambda_n\rightarrow\lambda$ for any $x\in{X}$, $(x_n)\subset{X}$, $\lambda\in\mathbb{K}$ and $(\lambda_n)\subset\mathbb{K}$.
\end{itemize}
The space $X$ equipped with $F$-norm is called an $F$-space. An $F$-space is called a Frech\'et space  if $X$ is a complete metric space with respect to the metric introduced by $F$-norm.
\end{defn}
\begin{defn}
 \label{def31}
 Let $ (T, \Sigma, \mu) $ be a measure space with a $ \sigma$-finite measure $ \mu$ and let $ (W, \| \cdot\|)$ be a Banach space. A function $s: T \rightarrow W$ is called a simple function if $ s = \sum_{i=1}^n w_i \chi(E_i)$ 
 where $ \{E_i\}_{i=1}^n$ is a partition of $T,$ and for $i=1,...,n,$ $ E_i \in \Sigma,$ $ \mu(E_i)>0,$ $w_i \in W.$ A function $f:T \rightarrow W$ is called measurable if there exists a sequence of simple functions $\{s_n\}$ such that
 $\lim_n s_n(t) = f(t)$ $\mu$-almost everywhere. By $ L_o(T)$ we denote the set of all measurable functions from $ T$ into $W$ (with equality $\mu$-almost everywhere).
\end{defn}
\begin{defn}
\label{def32}
Let $ (T, \Sigma, \mu) $ be a measure space with a $ \sigma$-finite measure $ \mu$ and let $(W,\norm{\cdot}{}{})$ be a Banach space. Let 
$$
S_F = \{ s \in L_o(T): s \hbox{ is a simple function }, \mu(supp(s)) \in [0, + \infty) \}. 
$$
A Frech\'et space $ (X, | \cdot|_F)$ is called $F$-admissible if the following conditions are satisfied:

a. $X \subset  L_o(T);$

b. $S_F \subset X,$

c. If $f, g \in X$ and $ \|f(t)\| \leq \|g(t)\| $ $\mu$-almost everywhere, then $f \in X$ and $ |f|_F \leq |g|_F;$

d. for any sequence $ \{w_n\} \subset W$ and $ A \in \Sigma, $ $ \mu(A) < \infty $ if $ \|w_n\| \rightarrow 0$ then $ |w_n \chi(A)|_F \rightarrow 0;$

e. for any sequence $ \{ A_n \} \subset \Sigma $ and $ w \in W \setminus \{0\}$ $ |w\chi(A_n)|_F \rightarrow 0$ iff $ \mu(A_n) \rightarrow 0;$  

f. if $ f \in X,$ $ \{ f_n\} \subset X,$ $f_n \rightarrow f$ $\mu$-almost everywhere and $ \|f_n(t)\| \leq \|f(t)\|$ for any $t \in T, $ then 
$$
|f_n-f|_F \rightarrow 0.
$$ 
\end{defn}
Now we state a version of the Jegoroff Theorem for $L_o(T).$
\begin{lem}
\label{lem31}
Let $ (T, \Sigma, \mu) $ be a measure space, $ 0<\mu(T) < \infty$ and let $ W$ be a Banach space. Let $\{f_n\} \subset L_o(T)$ satisfies the Cauchy condition with respect to the convergence $\mu$-almost everywhere. 
Then, for any $ m >0 $ there exists $ A_m \in \Sigma, $ $ \mu(A_m) < m $ such that for any $ \epsilon >0$ there exists $ n_o \in \mathbb{N}$ such that
$$
sup \{ \|f_q(t) - f_p(t)\|: t \in T \setminus A_m\} \leq \epsilon
$$
provided $ p,q \geq n_o.$ 
\end{lem}
\begin{proof}
Fix $ m > 0$ and $ k \in \mathbb{N} \setminus \{0\}.$ Define
$$
A_{i,j,k} = \{ t \in T: \| f_i(t) - f_j(t)\| < 1/k\} \hbox{ and }  B_{p,k} = \bigcap_{i+j \geq p} A_{i,j,k}. 
$$
Notice that for any $ p \in \mathbb{N}, $ $B_{p,k} \subset B_{p+1,k}.$ Moreover, 
$$
\bigcup_{p=1}^{\infty} B_{p,k} = T.
$$
Indeed, if $t \in T, $ then there exists $ p_o \in \mathbb{N}$ such that  for $ i, j \geq p_o, $ $ \|f_i(t) - f_j(t)\| < 1/k .$ Hence, $ t \in B_{p_o,k}.$ Since $ \mu(T) < \infty, $ for any $ k \in \mathbb{N} \setminus \{0\}$ there exists
$p(k)$ such that $ \mu (T \setminus B_{p(k),k}) < \frac{1}{2^k}.$ 
Fix $k_o \in \mathbb{N}$ such that $ \sum_{j=k_o}^{\infty} \mu(T \setminus B_{p(j),j}) < m$ and put
$$
A_m = \bigcup_{j=k_o}^{\infty} (T \setminus B_{p(j),j})
$$
It is clear that $ \mu(A_m) < m.$ Moreover $ T \setminus A_m=  \bigcap_{j=k_o}^{\infty} B_{p(j),j}.$ 
Now fix $ \epsilon > 0$ and $ j_o \in \mathbb{N}, $ $j_o \geq k_o$  such that $ \frac{1}{j_o} < \epsilon.$ Observe that if $ t \in T \setminus A_m$ then $ t \in B_{p(j_o),j_o}.$ Hence, for any $i,j \geq p(j_o)$
$$
sup\{ \|f_i(t) - f_j(t) \| : t \in T \setminus A_m \} < \frac{1}{j_o} < \epsilon, 
$$ 
which proves our claim.
\end{proof}
Reasoning in the same way as in Lemma \ref{lem31} we can prove
\begin{lem}
\label{lem32}
Let $ (T, \Sigma, \mu) $ be a measure space, $ 0<\mu(T) < \infty$ and let $ W$ be a Banach space. Let $\{f_n\} \subset L_o(T)$ and $ f \in L_o(T).$ assume that $ f_n(t) \rightarrow f(t) $ $\mu$-almost everywhere. 
Then, for any $ m >0 $ there exists $ A_m \in \Sigma, $ $ \mu(A_m) < m $ such that for any $ \epsilon >0$ there exists $ n_o \in \mathbb{N}$ such that
$$
sup \{ \|f_n(t) - f(t)\|: t \in T \setminus A_m\} \leq \epsilon
$$
provided $ n \geq n_o.$ 
\end{lem}
\begin{thm}
\label{thm31}
Let $ X \subset L_o(T)$ be an $F$-admissible $ F$-space. Assume that $ \mu(T) < \infty.$ Let for $ M >0,$ 
$$
W^M = \{ s \in S_F: sup \{\|s(t) \|: t \in T \} \leq M \}.
$$
Let $ K = \{ T_1,...,T_k\}$ be an arbitrary partition of $T$ and 
$$ 
S^K = \{ s \in S_F: s = \sum_{j=1}^k w_i \chi(T_i): w_i \in W \}.
$$ 
Define $ P_K: S \rightarrow  S^K$ by 
$$
P_K(s) = \sum_{i=1}^k z_i(s) \chi(T_i),
$$
where for $s = \sum_{j=1}^l w_j \chi(S_j), $ and $i=1,...,k,$ 
$$
z_i(s) = \frac{\sum_{j=1}^l w_j \chi(S_j \cap T_i)\mu(S_j \cap T_i)}{\mu(T_i)}.
$$
Assume that $ \{ s_n \} \subset W^M $ satisfies the Cauchy condition with respect to the convergence $ \mu$-almost everywhere.
The for any $ \epsilon > 0$ there exists $ n_o \in \mathbb{N}$ such that for $ n,m \geq n_o$ and any partition $K$ 
$$
|P_K(s_n) - P_K(s_m)|_F \leq \epsilon.
$$
\end{thm}
\begin{proof}
Fix $ \epsilon > 0.$ By Lemma \ref{lem31} for any $ p \in \mathbb{N} \setminus \{0\},$ there exists a sequence $\{ A_p\} \subset T $ such that $\mu(A_p) \rightarrow 0,$ and $\{ s_n\}$ satisfies the Cauchy condition with respect to the uniform convergence on $ T \setminus A_p.$ Fix a partition $K.$ Observe that for any $ p \in \mathbb{N} \setminus \{0\},$
$$
|P_K(s_n)- P_K(s_m)|_F = |P_K(s_n- s_m)\chi(T\setminus A_p) + P_K(s_n- s_m)\chi(A_p) |_F  
$$
$$
\leq |P_K(s_n- s_m)\chi(T\setminus A_p)|_F + |P_K(s_n- s_m)\chi(A_p) |_F. 
$$ 
Fix $ w \in W $ such that $ \|w \|= 2M.$ By $F$-admissibility of $X$ there exists $ p_o \in \mathbb{N} \setminus \{0\},$ such that $|w\chi(A_{p_o}) |_F \leq \epsilon/2.$  
Since for any $ n \in \mathbb{N}, $ $ s_n \in W^M, $ by defintion of $ P_K,$ $ P_K(s_n) \in W^M,$ too. Hence, for any $ t \in T, $ $n,m \in \mathbb{N}$ 
$$
\|P_K(s_n-s_m)\chi(A_{p_o}(t))\| \leq \|w \chi(A_{p_o})(t)\|.
$$
By $F$-admissibility of $X,$ for any $ n,m \in \mathbb{N}$ 
$$
|P_K(s_n-s_m)\chi(A_{p_o})|_F \leq |w \chi(A_{p_o})|_F < \epsilon/2.
$$
Now fix $ z \in W \setminus \{0\}.$ By $F$-admissibility of $X,$ there exists $k_o \in \mathbb{N}$ such that $ |\frac{z}{k_o}\chi(T)|_F < \epsilon/2.$
By defintion of $ A_{p_o}$ and Lemma \ref{lem31}, there exists $ n_o \in \mathbb{N}$ such that $ n,m \geq n_o $
$$ 
\sup \{ \|P_K(s_n-s_m)(t)\|:t \in T \setminus A_{p_o}\} \leq \frac{\|z\|}{k_o}.
$$
By $F$-admissibility of $X$ for $ n,m \geq n_o,$
$$
|P_K(s_n- s_m)\chi(T\setminus A_{p_o})|_F \leq |\frac{z}{k_o}\chi(T\setminus A_{p_o})|_F \leq |\frac{z}{k_o}\chi(T)|_F < \epsilon/2.
$$
Observe that for any partition $K$ and $ n, m \in \mathbb{N},$  
$$ 
sup \{\|P_K(s_n-s_m)(t)\|: t \in T \} \leq sup \{\|s_n-s_m\|:t \in T \}.
$$
Hence, the choice of $n_o$ is independent of $K.$ The proof is complete. 
\end{proof}
\begin{lem}
\label{lem33}
Let $ X \subset L_o(T)$ be an $F$-admissible $ F$-space. Assume that $ \mu(T) < \infty.$ Let $f \in X$ and let $sup\{ \|f(t)\|: t \in T\} =M < \infty.$
Then, there exists a sequence $ \{ w_n\} \subset W^{4M}$ such that $ w_n \rightarrow f$ $\mu$-almost everywhere. Moreover, $|w_n -f|_F \rightarrow 0.$ 
\end{lem}
\begin{proof}
Since $ f \in L_o(T), $ there esists a sequence $ \{ s_n\} \subset S_F $ such that $s_n \rightarrow f$ $\mu$-almost everywhere.
Let $ s_n = \sum_{j=1}^{k_n} s_{n,j}\chi_{A_{n,j}}.$ 
Set $ w_{n,j} = s_{n,j} $ if $ \|s_{n,j}\| \leq 4M$ and $ w_{n,j} = \frac{2M s_{n,j}}{\|s_{n,j}\|} $ in the opposite case.
Define
$$
w_n = \sum_{j=1}^{k_n} w_{n,j}\chi_{A_{n,j}}.
$$
Let $ t \in T.$ Then, there exists exactly one $j \in \{ 1,...,k_n\}$ such that $ t \in A_{n,j}.$ If $ \|s_{n,j}\| > 4M,$ then 
$$
\|(s_{n,j} - f)(t)\| - \|(w_{n,j} - f)(t)\|\geq \|s_{n,j}(t)\| - \|f(t)\| - (\|w_{n,j}(t)\|+\|f(t)\|) 
$$
$$
\geq \|s_{n,j}(t)\| - \|w_{n,j}(t)\| -2\|f(t)\| > 0.
$$
Hence, $w_n \rightarrow f$ $ \mu$-almost everywhere.
Now we show that $ |w_n -f|_F \rightarrow 0.$
Applying Lemma \ref{lem32} and reasoning as in Theorem \ref{thm31}, we get that for any $\epsilon >0$ and $ k \in \mathbb{N} \setminus \{0\}$ there exists $ A_k \subset T,$ 
and $ n_o \in \mathbb{N}$ such that for $ n \geq n_o$
$$
|(f-w_n)\chi(A_k)|_F \leq \epsilon/2 \hbox{ and } |(f-w_n)\chi(T \setminus A_k)|_F \leq \epsilon/2.
$$
Hence, $ |f-w_n|_F \rightarrow 0.$ 
\end{proof}
\begin{lem}
\label{lem33a}
Let $ X \subset L_o(T)$ be an $F$-admissible $ F$-space. Assume that $ \mu(T) < \infty.$ Assume that $ \{ f_n \} \subset X,$ $f \in X$  and $ |f_n - f|_F \rightarrow 0.$
Then, there exists a subsequence $ \{ n_k\}$ such that $f_{n_k} $ converges to $f$  $\mu$-almost everywhere. 
\end{lem}
\begin{proof}
Assume that $|f_n-f|_F \rightarrow 0.$ Fisrt we show that $f_n $ converges to $ f$ in measure. 
Assume that this is not true. Passing to a convergent subsequence, if necessary, we can assume that there exist $ d>0 $ and $ e> 0$ such that 
$$
\mu(A_{d,n}) = \mu(\{ t \in T: \|f_n(t) - f(t)\| \geq d \}) > e
$$
for all $n\in\mathbb{N}$. Fix $ z \in W$ such that $\|z\|=d$ and let $ g_n = z \chi(A_{d,n}).$
Note that for any $ t \in T, $
$$
\|(f-f_n)(t)\| \geq \|(f-f_n)\chi(A_{d,n})(t)\| \geq \|g_n(t)\|.
$$
Since $X$ is an $F$-admissible $F$-space  $|g_n|_F$ does not converge to $0$ and consequently $ |f-f_n|_F $ does not converge to $0;$ a contradiction.
Now, select for any $ k \in \mathbb{N} \setminus \{0\} $ $n_k \in \mathbb{N}$ such that $\mu(A_{1/k,n_k}) \leq \frac{1}{2^k}.$ Let for any $ m \in \mathbb{N},$
$B_m = \bigcup_{k=m}^{\infty} A_{1/k,n_k}.$ Notice that
$$
\mu(B_m) \leq \sum_{k=m}^{\infty}\mu( A_{1/k,n_k}) \leq \frac{1}{2^{m-1}}.
$$
Let $B = \bigcap_{m=1}^{\infty} B_m.$ It is clear that $\mu(B) = lim_m \mu(B_m) =0.$ Let $ t \in T \setminus B.$
Then, there exists $m_o$ such that $ t \notin B_{m_o}.$ Hence, for any $ k \geq m_o $ $ t \notin A_{1/k, n_k}.$ 
Consequently for any $ k \geq k_o$ $ \|f(t) - f_n(t) \| < \frac{1}{k_o},$ , which shows our claim.
\end{proof}
\begin{lem}
\label{lem34}
Let $ X \subset L_o(T)$ be an $F$-admissible $ F$-space. Assume that $ \mu(T) < \infty.$ Fix $f \in X$ with $sup\{ \|f(t)\|: t \in T\} =M < \infty.$
Let $ \{ w_n\} \subset W^{4M}$ be so chosen that $ w_n \rightarrow f$ $\mu$-almost everywhere. Let $ K $ be a fixed partition of $T$ 
and let $ P_K $ be the operator from Theorem \ref{thm31}. Then, there exists
$$
lim_{n} P_k(w_n)=P_K(f) 
$$
and it is independent of the choice of $\{ w_n\}.$
\end{lem}
\begin{proof}
Let $ \{w_n\} \subset W^{4M}$ converges to $f$ $\mu$-almost everywhere. (By Lemma \ref{lem33} such a sequence exists.
Then, $ \{w_n\}$ satisfies the Cauchy condition with respect to the $\mu$-almost everywhere convergence.
By Theorem \ref{thm31}, the sequence $ \{ P_K(w_n)\}$ satisfies the Cauchy condition with respect to the convergence in $X.$ Since $X$ is complete, 
there exists $ lim_n P_K(w_n).$ If $ \{z_n\} \subset W^{4M}$ is the other sequence  converging to $f$ $\mu$-almost everywhere, then considerning 
a sequence $s_{2n} = w_n$ and $s_{2n+1} = z_n,$ we get that the obtained limit is independent of the choice of $\{w_n\} \subset W^{4M}.$
\end{proof}
\begin{lem}
\label{lem35}
Let $ X \subset L_o(T)$ be an $F$-admissible $ F$-space. Assume that $ 0 < \mu(T) < \infty.$ Let $ \{ f_n\} \subset X$ and $ f \in X.$
Assume that there exists $M >0$ such that $ sup\{  \|f_n(t)\|: t \in T, n \in \mathbb{N} \} \leq M $ and $ sup\{\|f(t)\|: t \in T \} \leq M.$
Let $K$ be a fixed partition of $T.$ If  $ |f_n -f|_F \rightarrow 0,$ then $ |P_K(f_n)-P_K(f)|_F \rightarrow 0.$
\end{lem}
\begin{proof}
By definition of $ P_K, $ for any $ n \in \mathbb{N} $ there exists $ \{ s_{k,n}\} \subset W^{4M}$ such that $ | P_K(f_n) - P_K(s_{k,n})|_F \rightarrow_k 0.$
Hence, for any $ n \in \mathbb{N}$ there exists $ k_n \in \mathbb{N}$ such that                         
$$
|P_K(f_n) - P_K(s_{k_n,n})|_F \leq 1/(2n) \hbox{ and } |f_n - s_{k_n,n}|_F < 1/(2n). 
$$
Since $|f_n -f|_F \rightarrow, 0, $ $ |f - s_{k_{n},n}|_F \rightarrow_n 0.$ By definition of $ P_K,$
$$ 
|P_K(f) - P_K(s_{k_n,n})|_F \rightarrow_n 0.
$$
Hence, 
$$
0 \leq |P_K(f) - P_K(f_n)|_F \leq |P_K(f) - P_K(s_{k_n,n})|_F+ |P_K(f_n) - P_K(s_{k_n,n})|_F 
$$
$$
\leq 1/(2n)+ |P_K(f) - P_K(s_{k_n,n})|_F .
$$
Hence, $ |P_K(f_n) -P_K(f)|_F \rightarrow 0, $ as required.
\end{proof}
\begin{defn}
\label{def33}
Let $ 0< \mu(T) < \infty$ and let $ K =\{K_1,...K_n\} $ and $G= \{G_1,...,G_m\},$  denote two partitions of $T,$ 
$ K_i, G_j \in \Sigma,$ $ 0< \mu(K_i)$ for $i=1,...,n$ and $ 0 < \mu(G_j)  $ for $j=1,...,m.$
We say that $ F \leq  G$ if, for any $ i=1,...,n$,
$$
K_i = \sum_{j=1}^{m_i} G_{ij}
$$
with $ G_{ij} \in G $ for $ j=1,...,m_i$ and $1\leq m_i\leq m$.
\end{defn}
\begin{lem}
\label{lem36}
Let $ s \in W^M\setminus \{0\} $ and let $ 0< \mu(T) < \infty.$ Let $ K = \{ K_1,...,K_n\}$ be a partition associated with $s,$ i.e.
$$
s = \sum_{j=1}^n w_j \chi(K_j).
$$ 
If $G= \{G_1,...,G_k\} $ is a partition of $T$ and $ K \leq G$ then $ P_G(s) = s.$
\end{lem}
\begin{proof}
Notice that 
$$
P_G(s) = \sum_{i=1}^k z_i(s) \chi(G_i),
$$
where for $s = \sum_{j=1}^n w_j \chi(K_j), $ and $i=1,...,k,$ 
$$
z_i(s) = \frac{\sum_{j=1}^n w_j \chi(K_j \cap G_i)\mu(K_j \cap G_i)}{\mu(G_i)}.
$$ 
Since $K \leq G,$ there exists exactly one $j_i \in \{1,...,n\}$ such that $G_i \subset K_{j_i}.$
Hence, 
$$
z_i(s)= \frac{w_{j_i}\chi(G_i \cap K_{j_i})\mu(G_i\cap K_{j_i})}{\mu(G_i)} = w_{j_i}\chi(G_i).
$$
Let $ t \in T.$ Then, there exists exactly one $ i \in \{ 1,...,k\},$ such that $ t \in  G_i.$ Hence, $ s(t)= w_{j_i},$ which shows that $ P_G(s)=s.$ 
\end{proof}
\begin{lem}
\label{lem37}
Let $ 0< \mu(T) < \infty.$ Let for $ n \in \mathbb{N},$ $ K^n = \{ K_{1,n}...,K_{k,n}\}$ be a partition of $T.$ Then, there exists a sequence $ \{ S^n\}$ of partitions of $T$
such that $ S^i \leq S^j $ for $ i \leq j$ and for any $ n \in \mathbb{N},$ $ K^n \leq S^j$ for $ j \geq n.$ 
\end{lem}
\begin{proof}
Let $ S^1 = K^1. $ Put $ S^2 = \{ K_1 \cap K_2 : K_1 \in K^1, K_2 \in K^2\}.$ It is clear that $ S^2$ is a partition of $T$ and $ K^i \leq S^2$ for $ i=1,2.$
Now assume that we have constructed $ S^1,...,S^n.$ 
Define $S^{n+1} = \{ S \cap K: S \in S^n, K \in K^{n+1}\}.$ It is clear that $ S^{n+1} $ is a partition of $T,$ $ S^i \leq S^{n+1}$ for $ i =1,...,n$ and 
$K^i \leq S^{n+1}$ for $ i=1,...,n+1.$ By our construction, for any $ n \in \mathbb{N}$ and $ j \geq n,$ $ S^n \leq S^j.$ Also $ K^n \leq S^j,$ for $j \geq n,$
as required.  
\end{proof}
\begin{thm}
\label{thm33}
Let $ X \subset L_o(T)$ be an $F$-admissible $ F$-space. Assume that $ 0 < \mu(T) < \infty.$ Let $Z \subset X$ be a compact set.
Assume that there exists $M >0$ such that $sup\{\|f(t)\|: f \in Z, t \in T \} \leq M.$ Then, for any $ \epsilon > 0$ there exists $K$ a partition of $T$ such that
$$
sup\{ |P_K(f)-f|_F: f \in Z \} < \epsilon.
$$
\end{thm}
\begin{proof}
Since $Z$ is compact, there exists a countable, dense set $ A= \{ g_n \} \subset Z.$ Fix for any $n \in \mathbb{N}$ a sequence $ \{s_{k,n}\} \subset W^{4M}$
such that $ s_{k,n} \rightarrow_k g_n$ $ \mu$-almost everywhere. (By Lemma \ref{lem33}) such a sequence exists.) Let $\{ U_{k,n} \}$ be the sequence of partitions associated with
$\{s_{k,n}\}.$ Let $ \phi: \mathbb{N} \rightarrow  \mathbb{N} \times \mathbb{N}$ be a fixed bijection. Put $K_n = U_{\phi(n)}.$ By Lemma \ref{lem37} there exists a sequence
of partitions $\{ S_n\}$ such that $ S^i \leq S^j $ for $ i \leq j$ and for any $ n \in \mathbb{N},$ $ K^n \leq S^j$ for $ j \geq n.$
Let $P_n = P_{S_n}$ be the projection associated with $S_n$  by Theorem \ref{thm31}. We show that for any $ \epsilon > 0$ there exists $n \in \mathbb{N}$ such that
$$
sup\{ |P_n(f)-f|_F: f \in Z \} < \epsilon.
$$ 
Assume on the contrary that there exists $ \epsilon > 0 $ such that for any $ n \in \mathbb{N}$ there exists $ f_n \in Z$ such that
$$
|f_n- P_nf_n|_F \geq \epsilon. 
$$ 
Passing to a converging subsequence, if necessary, we may assume that $ |f_n-f|_F \rightarrow 0$ for some $ f \in Z. $ Again, by Lemma \ref{lem33a}, passing to a convergent subsequence, if necessary, we may assume that $ f_n \rightarrow f$ $\mu$-almost everywhere.
Let for any $ n \in \mathbb{N},$ $ \{ w_{k,n} \}\subset W^{4M}$ be such that $ |f_n -w_{k,n}|_F \rightarrow_k 0$ and $ \|w_{k,n}(t) - f_n(t)\| \rightarrow 0$ $\mu$-almost everywhere. Since $ \{ g_n\}$ is a dense subset of $Z,$ by Lemma \ref{lem33a}, 
we can asssume that 
for any $ n \in \mathbb{N},$ $ \{ w_{k,n}\} \subset \bigcup_{m=1}^{\infty} \{ s_{k,m}\}.$
Since $ f_n \rightarrow f$ $\mu$-almost everywhere and $ w_{k,n} \rightarrow_k f_n$ $\mu$-almost everywhere, by Lemma \ref{lem33a}, we can select for any $ n \in \mathbb{N}$ $w_{k_n,n}$ such that
$$
|f-w_{k_n,n}|_F \rightarrow_n 0, \ f(t) - w_{k_n,n}(t) \rightarrow 0 \ \mu\hbox{-almost everywhere ,}
$$
$ |P_n(f_n) - P_n(w_{k_n,n})| \leq \epsilon/8,$ 
and $ w_{k_n,n} $ satisfies the Cauchy condition with respect to the convergence $ \mu$-almost everywhere.
Fix $ n \in \mathbb{N}$ such that $|f_n-f|_F \leq \epsilon /8.$
$$
|f_n -P_n(f_n) |_F \leq |f-f_n|_F + |f-P_n(f)|_F + |P_n(f_n) - P_n(f)|_F. 
$$
By definition of $ P_n,$ $\{s_{k,n} \}$ and $ \{ w_{k,n}\} $ there exists $ m_o \in \mathbb{N} $ such that for $ m \geq m_o,$ $ |f-w_{k_m,m}|_F \leq \epsilon/8, $ $w_{k_m.m} = P_n(w_{k_m,m})$ and $ |P_nf -P_n(w_{k_m,m})| < \epsilon/8.$
Hence, 
$$
|f-P_nf|_F \leq |f-w_{k_m,m}|_F +|w_{k_m,m}-P_n(f)|_F 
$$
$$
= |f-w_{k_m,m}|_F +|P_n(w_{k_m,m})-P_n(f)|_F \leq \epsilon/4.
$$
Also 
$$
|P_n(f_n) - P_n(f)|_F \leq |P_n(f_n) - P_n(w_{k_n,n})|_F 
$$
$$
+ |P_n(w_{k_n,n}) - P_n(w_{k_m.m})|_F + |P_n(f)- P_n(w_{k_m,m})|_F 
$$
Since $\{w_{k_n, n }\}$ satisfies the Cauchy condition with respect to the $\mu$-almost everywhere convergence, by Theorem \ref{thm31}, for $ n,m \geq n_o$ 
$$
|P_n(w_{k_n,n}) - P_n(w_{k_m.m})|_F \leq \epsilon/8.
$$
Consequently, $|P_n(f_n)- f_n|_F < \epsilon$ for $n$ sufficiently large, a contradiction. 
\end{proof}

\section{Admissibility}
In this section we prove the admissibility of any $F$-admissible $F$-space.
We begin with the following proposition by showing that the space of simple functions generated by a fixed partition $K$ of $T$ is admissible.
\begin{prop}
\label{fin-dim}
Let $K =\{ K_1, \cdots, K_n\}$  be a finite partition of
$T$ such that $\mu(K_i) > 0$ for any $i \in \{1,...,n\}.$ Then, the subspace
\[
S_K=  \Big\{ s \in X : s= \sum_{i=1}^{n}w_{i}\chi(A_{i}),
\  \ w_{i}\in W \Big\}
\]
of $X$ is admissible.
\end{prop}
\begin{proof}
Let $Z$ be a compact subset of $S_K$ and $\eps >0$ be given.
For each $g \in Z$ we can write
\[ g= \sum_{i=1}^n w_i(g)  \chi_{A_i} \] for suitable elements
$w_i(g) $ of the Banach space $W$.

First we show that for any $ i=1,...,n$ the mapping $ g \rightarrow w_i(g)$
 is continuous with respect to $| \cdot |_F.$
Assume on the contrary that there exist $ i \in  \{1,...,n \}, $  $ g_k, g \in Z$ and $ d>0$ such that 
$ \|w_i(g_k) - w_i(g)\| \geq d. $ 
Fix $ z \in W, $ $ \|z \|=d. $ Let $ f = z \chi(A_i). $
Obserwe that for any $ t \in T, $  
$$
\|(g_k - g)(t)\| \geq \|(g_k-g)\chi(A_i)(t)\| \geq \|f(t)\|.
$$
Since $ \mu(A_i) >0, $ and $ X$ is an $F$-admissible $F$-space, $ |f|_F > 0, $ which leads to a contradiction.
Consequently, for any fixed $i=1, \cdots,n$, the set $ C_i =\{
w_i(g): \ g \in Z \}$ is a compact subset of $W$, and
$C=\bigcup_{i=1}^n C_i$ is a compact subset of $W$ too.\\
Let $\delta > 0 $ be fixed. Then, by the admissibility of the Banach space
$W$,
 there exist a finite dimensional space $Z_\delta = \mbox{span}[z_1, \cdots, z_m]$
 in $W$ and
a continuous mapping $H_\delta: \ C \to Z_\delta$ such that
   \begin{equation} \label{C}
   \|w - H_\delta (w) \| \le \delta \ \
    \mbox{for all} \ \ w\in   C.
  \end{equation}
  Then, for each $i \in \{1, \cdots,n\},$ $ g \in S_K$ and for suitable
  $w_j^i(g) \in \mathbb{R},$ $j=1,\cdots, m$ we can write
\[
  H_\delta(w_i(g))=    \sum_{j=1}^m
w_j^i(g)z_j.
  \]
As no confusion can arise we denote again by
     $H_\delta:  Z \to S_K$ the continuous mapping defined by
\[
 H_\delta (g)= \sum_{i=1}^n H_\delta (w_i(g)) \chi_{A_i} =
\sum_{i=1}^n  \Big(  \sum_{j=1}^mw_j^i(g) z_j   \Big) \chi_{A_i}.
\]
Then,
 $ H_\delta (S_K) \subseteq
   \mbox{ \rm span}[\chi_{A_i}z_j , i=1,\cdots, n;\ j=1,\cdots, m]$
   and  dim(span$[H_\delta (S_K)]) < + \infty$. On the other hand for
   each $g \in Z$ we have
\begin{eqnarray} \label{*}
 |g - H_\delta(g)|_F &= &|   \sum_{i=1}^n
w_i(g)\chi_{A_i} -
 \sum_{i=1}^n (  \sum_{j=1}^m w_j^i(g) z_j  )  \chi_{A_i}|_F
\\
& \le&   \sum_{i=1}^n | (w_i(g)  - \sum_{j=1}^m w_j^i(g)z_j
)\chi_{A_i} |_F . \nonumber
\end{eqnarray}
Since $X$ is an $F$-admissible $F$ space
\[
\sum_{i=1}^n | (w_i(g)  - \sum_{j=1}^m w_j^i(g)z_j
)\chi_{A_i}|_F \leq  \sum_{i=1}^n |\delta \chi_{A_i}|_F.
\]
Since $ | \cdot |_F$ is an F-norm, for any $ \eps > 0$ we can find $ \delta > 0$ such that $\sum_{i=1}^n |\delta \chi_{A_i}|_F <
\eps.$ This shows the admissibility of $ Z$ in $ X.$
\end{proof}

In order to prove our main result of this paper Theorem \ref{thm1} we need the following two lemmas.

\begin{lem}
\label{XinP}
Let $ Z \subset X $ be a compact set. Set  $F_n f = f \chi_{T_n},$ where $T= \bigcup_{n=1}^{\infty} T_n,$ $ \mu(T_n) < \infty, $ $ T_n \subset T_{n+1} $ for any $ n \in \mathbb{N}.$ 
Then, for any $f, g \in X$ and $ n \in \mathbb{N},$
\[
| F_n(f) -F_n(g)|_F \leq |f -g|_F.
\]
Moreover, for any $\eps > 0$ there exists $ n_o \in  \mathbb{N} $ such that for each $n \ge n_o$ we have
\[
\sup \{ | f - F_n(f)|_F: f \in Z \} < \eps .
\]
\end{lem}
\begin{proof}
Let $f,g \in X.$ Observe that for any $ t \in T$
\[
\|(f-g)(t)\| \geq \|F_n(f)(t)-F_n(g)(t)\|.
\]

Hence, by $F$-admissibility of $X,$
\[
| F_n(f) -F_n(g)|_F \leq |f -g|_F.
\]
Notice that for any $ f \in X$, $F_n(f) \rightarrow f$ $ \mu$-almost everywhere and $ \|F_n(f)(t)\| \leq \|f(t)\| $ for any $ t \in T. $ By $F$-admissibility of $X,$
$|f-F_n(f)|_F \rightarrow 0.$ 
Now assume that there exists $ \epsilon >0$ such that for any $n \in \mathbb{N}$ we have 
$$
|f_n - F_n(f_n)|_F \geq \epsilon.
$$
By compactness of $Z$ we can assume that $|f_n -f|_F \rightarrow 0, $ for some $f\in Z.$ Fix $ n _o \in \mathbb{N}$ such that for $ n \geq n_o,$ $ |f - F_{n}(f)|_F < \epsilon/3 $
and $|f-f_n|_F < \epsilon/3.$
Hence, for any $ n \geq n_o,$
$$
|f_n- F_n(f_n)|_F \leq |f_n-f|_F + |f-F_n(f_n)|_F 
$$
$$
\leq  |f_n-f|_F + |F_n(f)-F_n(f_n)|_F + |f-F_n(f)|_F 
$$
$$
\leq 2|f_n-f|_F + |f-F_n(f)|_F < \epsilon,
$$
a contradiction.
\end{proof}
Let $a>0$. Now we denote by
$R_a$  the radial projection, of the Banach space $W$ onto its closed ball $B_a(W)$ of radius $a$, defined for $w \in W$ by
\[
R_a w= \left \{ \begin{array}{llll} w & \mbox{if} \ \  \
 \|w\| \le a\\
 a \ \frac{w}{\|w\|}  & \mbox{if} \ \ \
 \|w\| >a.
\end{array}
\right.
\]
Then, we define  the mapping $T_a: X \rightarrow  X$ by
setting for $t \in T$
\[
(T_af)(t)= R_a(f(t)).\]
\begin{lem}
\label{radial}
For any $ a > 0$ and $ f, g \in X,$
\[
| T_af -T_ag|_F \leq 2| f -g|_F.
\]
Moreover, for any $\eps > 0$ and for any compact subset $Z$ of $X$ there exists $ a > 0 $ such that
\[
\sup \{ | f - T_af|_F: f \in Z \} < \eps .
\]
\end{lem}
\begin{proof}
Fix $ a > 0.$
Note that by definition of $T_a$ for any $ t \in T$
\[
\|T_af - T_ag\| \leq 2 \|f(t)-g(t)\|
\]
as $R_a$ is
a Lipschitz mapping with constant $2$ (see \cite{DW}).
Hence, by $F$-admissibility of $X,$
\[
| T_a(f) -T_a(g)\|_F \leq 2|f -g|_F.
\]
Now fix $ Z \subset X,$ $K$ compact and $ f \in Z.$
Note that for any $ t \in T, $ $ \lim_n (T_nf)(t) = f(t)$ and
$ \|(T_nf)(t)\| \leq \| f(t)\|$ for any $ n \in \mathbb{N}$. By $F$-admissibility of $X,$
\begin{equation}
\label{radial1}
| f - T_n(f)|_F \rightarrow 0.
\end{equation}
To prove our second assert assume by contradiction that there exists $ \eps > 0$ such that for any $n \in \mathbb{N}$ there exists $f_n \in Z$ such that
\[
| f_{n} - T_nf_{n}|_F > \eps.
\]
Without loss of generality, passing to a convergent subsequence if necessary, we can assume that
there exists $f \in X$ such that $ | f_{n} - f|_F \rightarrow 0.$
Fix $ n \geq n_o $ such that for $n \geq n_o$ $|f-f_n|_F < \epsilon/4$ and $ |f-T_n(f)|_F < \epsilon/4.$ Notice that for $ n \geq n_o,$
\[
\begin{array}{lllll}
\vspace{.2cm}
| f_{n} - T_nf_{n}|_F &\leq&
| f - T_n(f)|_F + | T_n(f) - T_n(f_{n})|_F +| f - f_{n}|_F
\\
 &\leq&  3 |f_n - f\|_F +| f - T_n(f)\|_F .
\end{array}
\]
Hence, $| f_{n} - T_n(f_{n})|_F < \eps$ for $n \geq n_o,$ a contradiction.
\end{proof}

 We now are in the position to prove our main
result.
\begin{thm} \label{thm1}
Any $F$-admissible $F$ space is admissible.
\end{thm}
\begin{proof}
 Fix $Z$ a compact set in $X$,
and $\varepsilon
>0$.
Since $Z$ is compact, by Lemma \ref{XinP}, there exists $n \in \mathbb{N}$ such that
\[
\sup\{ |F_nf - f |_F : f \in Z\} \leq \eps/4.
\]
Moreover, $F_n$ is continuous.
By Lemma \ref{radial} applied to the compact set $ F_n(Z),$ there exists $ a >0$ satisfying
\[
\sup\{ |T_aF_n(f) - F_n(f) |_F : f \in Z\} \leq \eps/4.
\]
Since $T_a$ is a continuous mapping, by Theorem \ref{thm33} applied to $(T_a\circ F_n)(Z)$ and $ T_n \in \Sigma,$ $ \mu(T_n) < \infty,$ there exists $ k\in \mathbb{N}$ with
\[
\sup\{ \|P_kT_aF_n(f) - T_aF_n(f) \|_{\rho} : f \in Z\} \leq \eps/4.
\]
Notice that by Lemma \ref{lem35}, $P_k $ is a continuous mapping for any $ k \in \mathbb{N}.$
Hence, by Proposition \ref{fin-dim}  applied to $W= (P_k \circ T_a\circ F_n)(Z)$ there exits $H_{\eps}: W \rightarrow E_{\rho}$
such that span$[H_{\eps}(W)]$ is finite-dimensional and
\begin{equation*}
\sup\{ \|H_{\eps}P_kT_aF_n(f) - P_kT_aF_n(f) \|_{\rho} : f \in Z\} \leq \eps/4.
\end{equation*}
Notice that the continuous mapping $ H= H_{\eps} \circ P_k \circ T_a\circ F_n$ satisfies dim[span$[H(Z)]\ < \infty .$ Moreover, by the above facts, for any $ f \in Z$
\[
| f - H(f)|_F \leq  | F_nf - f |_F + |T_a(F_nf) - F_n(f) |_F
\]
\[  + |P_kT_aF_n(f) - T_aF_n(f) |_F   +  |H_{\eps}P_kT_aF_n(f) - P_kT_aF_n(f) |_F \leq \eps.
\]
and  the admissibility of $X$ is proved. 
\end{proof}
Now we present two important consequences of admissibility and Theorem \ref{thm1}.
\begin{thm}
 \label{fixpoint}
Let $X$ be an admissible $F$-space. Let $T: X \rightarrow X$ be a compact and continuous mapping. Then, there exists $ f \in X$ such that $Tf=f.$ 
\end{thm}
\begin{proof} The proof which will be presented works for any admissible Hausdorff topological vector space and it is well-known. We present it for a sake of completeness.
Since $T$ is a compact mapping, the set $Z = cl[T(X)] \subset X$ is a compact set. Hence, by Theorem \ref{thm1}
for any $ \epsilon > 0 $ there exists a continuous mapping $ H_{\epsilon }: Z \rightarrow X$ such that dim[span$[H_{\epsilon}(Z)]\ < \infty $ and $\sup_{f \in Z} \|f - H_{\epsilon}f\|_{\rho} \leq \epsilon .$ 
Let $ T_{\epsilon } = H_{\epsilon} \circ T. $ Notice that $ T_{\epsilon}(X) =H_{\epsilon}(T(X))\subset H_{\epsilon}(Z)$ and consequently,
since $ conv(H_{\epsilon}(Z)) \subset X,$
$$ 
T_{\epsilon}[conv(H_{\epsilon}(Z))] \subset T_{\epsilon}(X) \subset H_{\epsilon}(Z) \subset conv(H_{\epsilon}(Z)).
$$ 
Also $T_{\epsilon}$ is a continuous map.
Since dim[span$[H_{\epsilon}(Z)]\ < \infty, $ by the Carath\'eodory Theorem, the set  $conv(H_{\epsilon}(Z))$ is a compact set. By the Brouwer Theorem
there exists $ f_{\epsilon} \in X$ such that $ T_{\epsilon} f_{\epsilon} = f_{\epsilon}.$
Hence, for any $ n \in \mathbb{N}, $ 
$$
| T f_{1/n} - f_{1/n} |_F = | T f_{1/n} - T_{1/n}f_{1/n} |_F = | T f_{1/n} - (H_{1/n} \circ T)f_{1/n} |_F \leq 1/n,
$$
since $ Tf_{1/n} \in Z.$ By the compactness of $ Z$ we can assume that $ \lim_n | T f_{1/n} - f |_F = 0 $ for some $f \in Z.$ Hence, by the above estimate  
$$ 
| f_{1/n} - f |_F  \leq | T f_{1/n} - f_{1/n} |_F + |f-Tf_{1/n}|_F.
$$   
Hence, $ lim_n |f_{1/n}-f|_F =0.$ 
By the continuity of $T$, $\lim_n | T f_{1/n} - Tf |_F = 0, $ which gives that $ Tf =f.$
\end{proof}
\begin{thm}
\label{fixpoint2}
Let $X$ be an admissible $F$ space and let $C \subset X$ be a nonempty set. Assume that $C$ is a continuous retract of $X,$ i.e. there exists a continuous mapping $ P:X \rightarrow C$ such that 
$Pc=c$ for any $ c \in C.$ Let $ T:C \rightarrow C$ be a continuous compact mapping. Then, there exist $x\in C$ such that $Tx=x.$
\end{thm}
\begin{proof}
Applying Theorem \ref{fixpoint} to the mapping $T \circ P$ we get that there exists $ x \in X $ such that $ (T \circ P)x =x.$ Since $ (T \circ P)x \in C, $ $x \in C$.
Consequently $ Px=x$ and $ Tx =x, $ as required.
\end{proof}
Now we present some examples of Banach spaces which fulfill the requirement of Definition \ref{def32}.
\begin{exa}
	Let $L^{0}(T,\Sigma,\mu)$ be the set of all (equivalence classes of) extended real valued measurable functions on $T$ with $\sigma$-finite measure space $(T,\Sigma,\mu)$. Let $E\subset L^0(T,\Sigma,\mu)$ be a Banach function space equipped with the norm $\Vert \cdot \Vert$ and $E_b$ be a closure of the set of all simple functions of $E$. Clearly, it is well known that $E_b\subset{E}$. By definition of the Banach function space we observe that the norm is monotone i.e. for any $x\in L^0$, $y\in E$ with $|x|\leq|y|$ a.e., we get $x\in E$ and $\|x\|_E\leq\|y\|_E$. Furthermore, for any set $A\in\Sigma$ with $\mu(A)<\infty$ and for any sequence $(t_n)\subset\mathbb{R}$ such that $|t_n|\rightarrow{0}$ as $n\rightarrow\infty$ we have $\|t_n\chi_{A}\|=|t_n|\|\chi_A\|\rightarrow{0}$ as $n\rightarrow\infty$. Assuming additionally $E$ is order continuous and taking a sequence $(A_n)\subset\Sigma$ and $t\in\mathbb{R}\setminus{0}$ we conclude that  $\mu(A_n)\rightarrow{0}$ as $n\rightarrow\infty$ if and only if $\|t\chi_{A_n}\|=|t|\|\chi_{A_n}\|\rightarrow{0}$ as $n\rightarrow\infty$. Furthermore, for any $(f_n)\subset{E}$ and $f\in{E}$ such that $f_n\rightarrow{f}$ a.e. and $|f_n(t)|\leq|f(t)|$ for any $t\in{T}$, then $\|f_n-f\|\rightarrow{0}$ as $n\rightarrow\infty$ (for more details see \cite{BS,KPS}).
\end{exa}

\begin{exa}
	Let $L^{0}(T,X)$ be the set of all vector valued measurable functions on $T$ with $\sigma$-finite measure space $(T,\Sigma,\mu)$. Let $E(X)\subset L^0(T,X)$ be a K\"othe-Bochner space equipped with the norm $\Vert \cdot \Vert_{E(X)}=\|\|\cdot\|_X\|_E$ and $S_F(X)=\{s\in{L^0(T,X)}:s\textnormal{ is a simple function with a finite and positive support}\}$. Obviously $S_F(X)\subset E(X)$. Since $E$ is a Banach function space, it follows that the norm is monotone. Consequently, for any $f,g\in E(X)$ with $\|f(\cdot)\|_X\leq\|g(\cdot)\|_X$ a.e., we have $\|f\|_{E(x)}\leq\|g\|_{E(X)}$. Moreover, for any set $A\in\Sigma$ with $\mu(A)<\infty$ and for any sequence $(x_n)\subset{X}$ such that $\|x_n\|_X\rightarrow{0}$ as $n\rightarrow\infty$ we get $\|x_n\chi_{A}\|_{E(X)}=\|x_n\|_X\|\chi_A\|_E\rightarrow{0}$ as $n\rightarrow\infty$. Now, suppose that $E$ is order continuous. Then, taking $x\in{X}\setminus{0}$ and $(A_n)\subset{T}$ we obtain $\mu(A_n)\rightarrow{0}$ as $n\rightarrow\infty$ if and only if $\|x\chi_{A_n}\|_{E(X)}=\|x\|_X\|\chi_{A_n}\|_E\rightarrow{0}$ as $n\rightarrow\infty$. Finally, assuming that $f\in E(X)$ and $(f_n)\subset{E(X)}$ with $\|f_n(\cdot)-f(\cdot)\|_X\rightarrow{0}$ a.e. and $\|f_n(t)\|_X\leq\|f(t)\|_X$ for all $t\in{T}$, then by order continuity of $E$ we conclude that $\|f_n-f\|_{E(X)}\rightarrow{0}$ as $n\rightarrow\infty$ (for more details see \cite{megg}).
\end{exa}
\section{Applications to modular spaces}
In this section we show that a large class of modular spaces satisfies the requirements of Definition \ref{def32}.
First we recall a notion of modular.
\newline
Let $\mathbb{C}$, $\mathbb{R}$, $\mathbb{R}^+$ and $\mathbb{N}$ be the sets of complex, reals, nonnegative reals and positive integers, respectively. Let us denote by $(e_i)_{i=1}^n$ a standard basis in $\mathbb{R}^n$.
% semimodular space
Let $ X$ be a linear space over $\mathbb{K}$, where $ \mathbb{K} = \mathbb{R}$ or   $ \mathbb{K} = \mathbb{C}.$ 
A function $ \rho: X \rightarrow [0, + \infty]$ is called a \textit{semimodular} if there holds for arbitrary $ x, y \in X:$
\begin{itemize}
	\item[$(a)$] $\rho(dx) = 0$ for any $ d \geq 0$ implies $x=0$ and $\rho(0)=0;$
	\item[$(b)$] $\rho(dx) = \rho(x),$ for any $ d \in \mathbb{K}, $ $ |d|=1;$
	\item[$(c)$] $\rho(ax+by) \leq \rho(x) + \rho(y)$
	for any $ a,b \geq 0,$ $a +b =1.$
\end{itemize}
If we replace (c) by:

(c1)  there exists $s \in (0,1]$ such that $\rho(ax+by) \leq a^{1/s}\rho(x) + b^{1/s}\rho(y)$
	for any $ a,b \geq 0,$ $a^{1/s} +b^{1/s} =1$
\newline
then $\rho$ is called $s$-convex (convex if $s=1.)$
\newline 
Let $X_{\rho}= \{ x \in X: \rho(tf)<\infty: \hbox { for some } t \in \mathbb{R}\}.$ Then, $X_{\rho}$ is called a modular space. 

\begin{thm}
\label{increasing}
 Let $ X$ be a linear space over $\mathbb{K},$ $ \mathbb{K} = \mathbb{R}$ or   $ \mathbb{K} = \mathbb{C}.$ Fix $ n \geq 2$ and let $\rho_i$ be a semimodular defined on $X$ for any $i\in\{1,\dots,n-1\}$. Put $\rho=\max_{1\leq i\leq{n-1}}\{\rho_i\}.$
 Assume that $ f: \mathbb{R}^n \rightarrow [0, +\infty)$ is an F-norm such that for any  $ x= (x_1,x_2,\dots,x_{n}) \in (\mathbb{R}_+)^{n}$ and $ y=(y_1,y_2,\dots,y_{n}) \in (\mathbb{R}_+)^{n}$
 if $ x_j \leq y_j$ for $j=1,\dots,n,$ then
 \begin{equation}
 \label{crucial}
 f(x) \leq f(y).
 \end{equation}
 Let us define for $ x \in X_{\rho},$
 $$
 |x|_f = \inf_{k>0}\left\{ f\left(ke_1+\sum_{i=2}^n\rho_{i-1}\left(\frac{x}{k}\right)e_i \right) \right\}.
 $$
 Then, $ | \cdot |_f$ is an $F$-norm in $ X_{\rho}.$
\end{thm} 
\begin{proof}
Since $ x \in X_{\rho},$ there exists $u >0$ such that $ \rho_i(x/u) < \infty,$ for $ i =1,...,n-1.$ Hence, $ |x|_f < \infty.$
Obviously, $ \|0\|_f=0.$
Now we show that $ |x|_f > 0$ for $ x \neq 0.$
First assume that $ \rho(x) \in \{ 0,+ \infty\}.$ Since $ x \neq 0, $ $ \rho(x/k_o) = + \infty $ for some $ k_o >0.$
Notice that for $ k \geq k_o,$ 
$$
f\left(ke_1+\sum_{i=2}^n\rho_{i-1}\left(\frac{x}{k}\right)e_i \right)\geq f(k_o,0,...,0)  >0.
$$
If $k < k_o,$ then
$$
f\left(ke_1+\sum_{i=2}^n\rho_{i-1}\left(\frac{x}{k}\right)e_i \right)\geq f\left(\sum_{i=2}^n\rho_{i-1}\left(\frac{x}{k_o}\right)e_i \right) = + \infty.
$$
If there exists $k_o >0, $ such that $ 0 < \rho(x/k_o) < \infty,$ then for $k < k_o,$
$$
f\left(ke_1+\sum_{i=2}^n\rho_{i-1}\left(\frac{x}{k}\right)e_i \right)
$$
$$
\geq f\left(ke_1+\sum_{i=2}^n\rho_{i-1}\left(\frac{x}{k_o}\right)e_i \right) \rightarrow_k f\left(\sum_{i=2}^n\rho_{i-1}\left(\frac{x}{k_o}\right)e_i \right) >0.             
$$
If $ k \geq k_o,$ then
$$
f\left(ke_1+\sum_{i=2}^n\rho_{i-1}\left(\frac{x}{k_o}\right)e_i \right) \geq f(k_oe_1,0,...,0) >0.
$$
Hence, $ \|x\|_f >0.$
By defintion of modular, $ \|ax\|_f = \|x\|_f$ for $ a \in \mathbb{K}, $ $ |a|=1.$
Now we show that $ |x+y|_f \leq |x|_f + |y|_f.$ To do that, fix $ \epsilon > 0,$ $u>0$ and $ v>0$ such that
$$
f\left(ue_1+\sum_{i=2}^n\rho_{i-1}\left(\frac{x}{u}\right)e_i \right) \leq |x|_f + \epsilon 
$$ 
and 
$$
f\left(ve_1+\sum_{i=2}^n\rho_{i-1}\left(\frac{y}{v}\right)e_i \right) \leq |y|_f + \epsilon. 
$$ 
Notice that
$$
f\left((u+v)e_1+\sum_{i=2}^n\rho_{i-1}\left(\frac{x+y}{u+v}\right)e_i \right) 
$$
$$
= f\left((u+v)e_1+\sum_{i=2}^n \rho_{i-1}\left(\frac{ux}{(u+v)u} +\frac{vy}{(u+v)v}\right) e_i \right)
$$
$$
\leq f\left((u+v)e_1+\sum_{i=2}^n \left( \left(\rho_{i-1}\left(\frac{x}{u}\right) +\rho_{i-1}\left( \frac{y}{v}\right)\right) e_i \right) \right)
$$
$$
\leq f\left(ue_1+\sum_{i=2}^n\rho_{i-1}\left(\frac{x}{u}\right)e_i \right)+f\left(ve_1+\sum_{i=2}^n\rho_{i-1}\left(\frac{y}{v}\right)e_i \right)
$$
$$
\leq |x|_f + |y|_f +2\epsilon.
$$
Hence, $ |x+y|_f \leq |x|_f + |y|_f, $ as required.
\end{proof}

Now, applying Theorem \ref{increasing} we get the following theorem.

\begin{thm}
 \label{increasing1}
Let $ \rho_1$ and $f: \mathbb{R}^2 \rightarrow \mathbb{R}$ be as in Theorem \ref{increasing} and let $ \rho_1$ be a semimodular. Then, the function
$$
 |x|_f = \inf_{k>0} \{ f(ke_1 +\rho_1(x/k)e_2) \}
 $$
is an $F$-norm on $ X_{\rho_1}.$
\end{thm}
\begin{proof}
It is necessary to apply Theorem \ref{increasing} for $n=2$ 
\end{proof}
\begin{rem}
\label{Luxemburg}
Observe that if $f$ is equal to the maximum norm on $ \mathbb{R}^2,$ $ \rho$ is a semimodular and $ x \in X_{\rho},$ then 
$$ 
| x |_f = \inf \{u>0: \rho(x/u) \leq u \},
$$
which means that $ | \cdot |_f$ coincides with the classical Luxemburg F-norm on $ X_{\rho}.$
\newline
Indeed, let $f(u,v) = max \{|u|, |v|\}.$
Let 
$$
|x|_L = inf \{ u>0: \rho(x/u) \leq u\}.
$$
We show that $ |x|_f = |x|_L.$
Notice that for any $ \epsilon >0, $ 
$$
\rho\left(\frac{x}{|x|_L + \epsilon}\right) \leq |x|_L + \epsilon.
$$
Hence, $ f(|x|_L+\epsilon, \rho(\frac{x}{|x|_L + \epsilon})) = |x|_L + \epsilon$, which shows that $ |x|_f \leq |x|_L.$ If $ |x|_f < |x|_L$ for some
$x \in X_{\rho},$ then there exist $ u> 0 $ and $ \delta >0$ such that
$$
f(u, \rho(x/u)) < |x|_L - \delta.
$$
This shows that $ |u|\leq  \|x\|_L - \delta$ and consequently
$$
\rho\left(\frac{x}{|x|_L -\delta}\right) \leq \rho\left(\frac{x}{u}\right) < |x|_L - \delta,
$$
which leads to a contradiction with definition of $ | \cdot |_L.$
\end{rem}
It is worth noticing that 
recently, in \cite{CieLew-Semi} the similar problems, devoted to $s$-norms generated by $s$-convex semimodulars, have been investigated. We present all details of them for sake of completeness and reader’s convenience. Namely, the following results were established.

\begin{thm}{\cite{CieLew-Semi}}\label{thm:increasing}
	Let $ X$ be a linear space over $\mathbb{K},$ $ \mathbb{K} = \mathbb{R}$ or   $ \mathbb{K} = \mathbb{C}$ and let $ s \in (0,1].$ Fix $ n \geq 2$ and let $\rho_i$ be a $s$-convex semimodular defined on $X$ for any $i\in\{1,\dots,n-1\}$. Put $\rho=\max_{1\leq i\leq{n-1}}\{\rho_i\}.$
	Assume that $ f: \mathbb{R}^n \rightarrow [0, +\infty)$ is a convex function such that $f(x)=0$ if and only if $ x=0.$ Assume furthermore that for any  $ x= (1,x_2,\dots,x_{n}) \in (\mathbb{R}_+)^{n}$ and $ y=(1,y_2,\dots,y_{n}) \in (\mathbb{R}_+)^{n}$
	if $ x_j \leq y_j$ for $j=2,\dots,n,$ then
	\begin{equation}
	f(x) \leq f(y).
	\end{equation}
	Let us define for $ x \in X_{\rho},$
	$$
	\|x\|_f = \inf_{k>0}\left\{k f\left(e_1+\sum_{i=2}^n\rho_{i-1}\left(\frac{x}{k^{1/s}}\right)e_i \right) \right\}.
	$$
	Then, $ \| \cdot \|_f$ is an $s$-norm (norm if $s=1)$ in $ X_{\rho}.$
\end{thm} 

\begin{thm}{\cite{CieLew-Semi}}
	Let $ \rho_1$ and $f: \mathbb{R}^2 \rightarrow \mathbb{R}$ be as in Theorem \ref{thm:increasing} and let $ \rho_1$ be an $s$-convex semimodular. Then, the function
	$$
	\|x\|_f = \inf_{k>0} \{k f(1, \rho_1(x/k^{1/s})) \}
	$$
	is an $s$-convex norm (norm if $s=1)$ on $ X_{\rho_1}.$
\end{thm}

It is worth recalling one more result that was presented in \cite{CieLew-Semi} and comparing with Remark \ref{Luxemburg}. 
\begin{rem}{\cite{CieLew-Semi}}
	Observe that if $f$ is equal to the maximum norm on $ \mathbb{R}^2,$ $ \rho$ is a convex semimodular and $ x \in X_{\rho},$ then 
	$$ 
	\| x \|_f = \inf \{u>0: \rho(x/u) \leq 1 \},
	$$
	which means that $ \| \cdot \|_f$ coincides with the classical Luxemburg norm on $ X_{\rho}.$
	If $f$ is equal to the $l_1$-norm on $ \mathbb{R}^2$ and $ \rho$ is a convex semimodular then $ \| \cdot \|_f$ is equal to the classical Orlicz-Amemiya norm on $ X_{\rho},$, which shows that the notion 
	of $ \| \cdot \|_f$ is a natural generalization of two classical norms 
	considered in semimodular spaces.
	Moreover, if $ f$ is the $l_p$-norm on $ \mathbb{R}^2, $ $1 < p < \infty$ and $ \rho$ is a convex semimodular then $ \| \cdot \|_f$ is equal to the $p$-Orlicz-Amemiya norm on $ X_{\rho}$ (see \cite{CuHuWi}).
\end{rem}
Now we state the following
\begin{thm}
\label{final}
Let $W$ be a Banach space and let $ (T, \Sigma, \mu) $ be a measure space with $ \sigma$-finite measure $\mu.$ Let $ L_o(T)$ denote the space af all $\mu$-measurable functions
going from $T$ into $W.$ Let $ \rho$ be a modular defined on $L_o(T)$ and let $X_{\rho}$ be a modular space.
Assume that $\rho$ satisfies the requirements of Definition \ref{def32} (instead of $|\cdot|_F.)$ Then, $(X_{\rho},|\cdot |_f)$ is an admissible F-space for any $F$ norm $|\cdot |_f$ given
in Theorem \ref{increasing1} and Theorem \ref{thm:increasing}. 
\end{thm}
\begin{proof}
Notice that if $ \{w_n\} \subset W,$ $ \|w_n\| \rightarrow 0,$ and $ A \in \Sigma, $ with $ \mu(A)<\infty, $ then for any $u>0,$  $ \rho((w_n \chi(A))/u) \rightarrow 0.$ By definition of 
$ | \cdot |_f$ this shows that $ |w_n\chi(A)|_f \rightarrow 0.$
Analogously if $ w \in W \setminus \{0\} $ and $\{A_n \} \subset \Sigma,$ then for any $ u>0$ $\rho((w\chi(A_n))/u) \rightarrow 0$ if and only if $\mu(A_n) \rightarrow 0$ and again by definition of $ |\cdot |_f,$ the same condition is satisfied by $ |\cdot |_f.$
\newline
If $ x \in X,$ $ \{ x_n\} \subset X,$ $x_n \rightarrow x$ $\mu$-almost everywhere and $ \|x_n(t)\| \leq \|x(t)\|$ for any $t \in T, $ then for any $ u>0,$  
$\rho((x_n-x)/u) \rightarrow 0, $ which shows that $|x_n-x|_f \rightarrow 0$ by definition of $|\cdot |_f.$
This shows that the assumptions of Theorem \ref{thm1} are satisfied by $(X_{\rho} | \cdot |_f).$
\end{proof} 
\begin{rem}
\label{examples}
Typical examples of modulars satisfying the requirements of Theorem \ref{final} are vector-valued Musielak-Orlicz spaces determined by $\sigma$-finite measure spaces and modular function spaces (see
\cite{Koz}).  It is clear that in general these spaces are not locally convex, (see e.g. \cite{Koz}, 
Theorem 3.4.1, p. 77). Also Theorem \ref{thm1} generalizes the main result (\cite{CL}) concerning modular function spaces, because to apply  Theorem \ref{thm1} to modular function spaces, the $\Delta_2$-condition is not needed. 
\end{rem}

\section{metric approximation property}

In this section we discuss criteria which guarantee that a rearrangement invariant Banach function space possesses metric approximation property. First, let us mention some basic notions and definitions which are used in our further investigation. Recall that an operator $T$ is said to be \textit{approximable} if it is the limit in the operator norm of the sequence of the finite rank operators. We say that a Banach space $X$ has the \textit{metric approximation property} if for any $\epsilon>0$ and any compact subset $K$ in $X$ there exists a finite rank operator $T$ from $X$ into $X$ such that $\norm{T}{}{}\leq{1}$ and $\norm{x-Tx}{X}{}<\epsilon$ for any $x\in{K}$.

We define as usual by $\mu$ the Lebesgue measure on $I=[0,\alpha)$, where $\alpha =1$ or $\alpha =\infty$, and by $L^{0}$ the set of all (equivalence classes of) extended real valued Lebesgue measurable functions on $I$. We use the notation $A^{c}=I\backslash A$ for any measurable set $A$. A Banach lattice $(E,\Vert \cdot \Vert _{E})$ is called a \textit{Banach function space} (or a \textit{K\"othe space}) if it is a sublattice of $L^{0}$ satisfying the following conditions
\begin{itemize}
	\item[(1)] If $x\in L^0$, $y\in E$ and $|x|\leq|y|$ a.e., then $x\in E$ and $\|x\|_E\leq\|y\|_E$.
	\item[(2)] There exists a strictly positive $x\in E$.
\end{itemize}
For simplicity of our notation we use the symbol $E^{+}={\{x \in E:x \ge 0\}}$. Recall that an element $x\in E$ is said to be a \textit{point of order continuity} if for any sequence $(x_{n})\subset{}E^+$ such that $x_{n}\leq \left\vert x\right\vert 
$ and $x_{n}\rightarrow 0$ a.e. we have $\left\Vert x_{n}\right\Vert
_{E}\rightarrow 0.$ Given a Banach function space $E$ is called \textit{order continuous 
}(shortly $E\in \left( OC\right) $) if every element $x\in{}E$ is a point of order continuity (see \cite{LinTza}). A Banach function space $E$ has the \textit{Fatou property} if for any $\left( x_{n}\right)\subset{}E^+$, $\sup_{n\in \mathbb{N}}\Vert x_{n}\Vert
_{E}<\infty$ and $x_{n}\uparrow x\in L^{0}$, then we get $x\in E$ and $\Vert x_{n}\Vert _{E}\uparrow\Vert x\Vert
_{E}$. Unless it is said otherwise, it is assumed that a Banach function space $E$ has the Fatou property. We denote the \textit{distribution function} for any function $x\in L^{0}$ by 
\begin{equation*}
d_{x}(\lambda) =\mu\left\{ s\in [ 0,\alpha) :\left\vert x\left(s\right) \right\vert >\lambda \right\},\qquad\lambda \geq 0.
\end{equation*}
For any element $x\in L^{0}$ we define its \textit{decreasing rearrangement} as follows
\begin{equation*}
x^{\ast }\left( t\right) =\inf \left\{ \lambda >0:d_{x}\left( \lambda
\right) \leq t\right\}, \text{ \ \ } t\geq 0.
\end{equation*}
Additionally, we employ the convention $x^{*}(\infty)=\lim_{t\rightarrow\infty}x^{*}(t)$ if $\alpha=\infty$ and $x^*(\infty)=0$ if $\alpha=1$. We denote for any function $x\in L^{0}$ the \textit{maximal function} of $x^{\ast }$ by 
\begin{equation*}
x^{\ast \ast }(t)=\frac{1}{t}\int_{0}^{t}x^{\ast }(s)ds.
\end{equation*}
For any $x\in L^{0}$, it is well known that $x^{\ast }\leq x^{\ast \ast },$ $x^{\ast \ast }$ is decreasing, continuous and subadditive. Furthermore, recall that two functions $x,y\in{L^0}$ are called \textit{equimeasurable} (shortly $x\sim y$) if $d_x=d_y$. Given a Banach function space $(E,\Vert \cdot \Vert_{E}) $ is called \textit{symmetric} or \textit{rearrangement invariant} (r.i. for short) if whenever $x\in L^{0}$ and $y\in E$ such that $x \sim y,$ then $x\in E$ and $\Vert x\Vert_{E}=\Vert y\Vert _{E}$. For a symmetric space $E$ we define $\phi_{E}$ its \textit{fundamental function} given by $\phi_{E}(t)=\Vert\chi_{(0,t)}\Vert_{E}$ for any $t\in [0,\alpha)$ (see \cite{BS}). Given $x,y\in{}L^{1}+L^{\infty }$ we define the \textit{Hardy-Littlewood-P\'olya relation} $\prec$ by 
\begin{equation*}
x\prec y\Leftrightarrow x^{\ast \ast }(t)\leq y^{\ast \ast }(t)\text{ for
	all }t>0.\text{ }
\end{equation*}

% $K$-monotonicity
Now, let us recall that a symmetric space $E$ is said to be $K$-\textit{monotone} (shortly $E\in(KM)$) if for any $x\in L^{1}+L^{\infty}$ and $y\in E$ with $x\prec y,$ then $x\in E$ and $\Vert x\Vert_{E}\leq \Vert y\Vert _{E}.$ It is commonly known that a symmetric space is $K$-monotone if and only if $E$ is exact
interpolation space between $L^{1}$ and $L^{\infty }.$ Let us also recall that a symmetric space $E$ equipped with an order continuous norm or with the Fatou property is $K$-monotone. Given a Banach function space $E$ is said to be \textit{reflexive} if $E$ and its associate space $E'$ are order continuous. We refer the reader for more details to see \cite{BS,KPS}.

Now, we present the main theorem of this section.

\begin{thm}\label{thm:approx:1}
	Let $E$ be an order continuous symmetric space on $I$. Then, $E$ has the metric approximation property. 
\end{thm}

\begin{proof}
	For the sake of completeness and reader’s convenience we present all details of the proof of the following theorem. In some parts of the proof we use similar technique presented in \cite{Ryan} for the space $L^p$. First, we define the finite rank operators $T_{\mathcal{A}}$ from $E$ into $E$ by
	\begin{equation*}
	T_{\mathcal{A}}(x)=\sum_{j=1}^n\left(\frac{1}{\mu(A_j)}\int_{A_j}xd\mu\right)\chi_{A_j}
	\end{equation*}
	for every $x\in{E}$, where $\mathcal{A}=\{A_1,\dots,A_n\}$ is a finite collection of pairwise disjoint measurable subsets of $I$, with finite positive measure. Next, taking $\mathcal{B}=\{B_j:j\in J\}$ any countable collection of pairwise disjoint subsets of $I$ with finite positive measure and assuming that $\Omega=I\setminus\bigcup_{j\in J}{B_j}$, by Theorem 4.8 \cite{BS} it is well known that the following operator 
	\begin{equation*}
	S_{\mathcal{B}}(x)=x\chi_{\Omega}+\sum_{j\in J}\left(\frac{1}{\mu(B_j)}\int_{B_j}xd\mu\right)\chi_{B_j}
	\end{equation*}
	is a contraction in $E$, i.e. for any $x\in{E}$ we have $\norm{S_{\mathcal{B}}(x)}{E}{}\leq\norm{x}{E}{}.$ Furthermore, letting that $\mathcal{A}$ is a finite subcollection of the collection $\mathcal{B}$ we can easily observe that $T_A(x)\prec{S_{\mathcal{B}}(x)}$ for any $x\in{E}$. Hence, since $E$ is symmetric it follows that $$\norm{T_{\mathcal{A}}(x)}{E}{}\leq\norm{S_{\mathcal{B}}(x)}{E}{}\leq\norm{x}{E}{}$$ for any $x\in{E}$. Therefore, for any finite subcollection $\mathcal{A}$ of the collection $\mathcal{B}$ we have $\norm{T_{\mathcal{A}}}{}{}\leq{1}$. Next, considering $\mathcal{A}$, $\mathcal{A'}$ any two finite collections of pairwise disjoint measurable subsets of $I$ with finite positive measure we introduce a relation $\preceq$ as follows $\mathcal{A}\preceq\mathcal{A'}$ if every set of $\mathcal{A}$ is the union of a subcollection of $\mathcal{A'}$. So, we obtain a directed set $\mathcal{D}$ with a preorder $\preceq$. Now, we show that for any $x\in{E}$ the net $(T_{\mathcal{A}}(x))$ converges to $x$ in $E$. First, by assumption that $E$ is order continuous it is enough to prove that the wanted convergence is satisfied for any simple function in $E$. Therefore, taking a simple function $x\in{E}$ such that
	\begin{equation*}
	x=\sum_{j=1}^nc_j\chi_{A_j},
	\end{equation*} 
	where $\{c_1,\dots,c_n\}\subset\mathbb{R}$ and $\mathcal{A}=\{A_1,\dots,A_n\}$ is a collection of pairwise disjoint measurable subset of $I$ with finite positive measure, we conclude that $T_{\mathcal{A'}}(x)=x$ for $\mathcal{A'}$ any finite collection of such sets with $\mathcal{A}\preceq\mathcal{A'}$. Let $x\in{E}$ and $(x_n)\subset{E}^{+}$ be a sequence of simple functions such that $0\leq x_n\leq|x|$ and $x_n\rightarrow{x}$ a.e. on $I$. Therefore, since $E$ is order continuous this yields that $\norm{x_n-x}{E}{}\rightarrow{0}$. Next, since the net $(T_{\mathcal{A'}})$ is uniformly bounded for any $\mathcal{A'}\in\mathcal{D}$ we get
	\begin{equation*}
	\norm{T_{\mathcal{A'}}(x-x_n)}{E}{}\leq\norm{x-x_n}{E}{}.
	\end{equation*}
	In consequence, we have
	\begin{align}\label{equ:1:approx:1}
	\norm{T_{\mathcal{A'}}(x)-x}{E}{}&\leq\norm{T_{\mathcal{A'}}(x)-T_{\mathcal{A'}}(x_n)}{E}{}+\norm{T_{\mathcal{A'}}(x_n)-x}{E}{}\\ 
	&=\norm{T_{\mathcal{A'}}(x-x_n)}{E}{}+\norm{T_{\mathcal{A'}}(x_n)-x}{E}{}\nonumber\\
	&\leq\norm{x-x_n}{E}{}+\norm{T_{\mathcal{A'}}(x_n)-x}{E}{}\nonumber
	\end{align}
	for any $\mathcal{A'}\in\mathcal{D}$. Let $\epsilon>0$. Then, there exists $n_0\in\mathbb{N}$ such that for all $n\geq{n_0}$ we have 
	\begin{equation}\label{equ:2:approx:1}
	\norm{x-x_n}{E}{}<\epsilon/2.
	\end{equation}
	Next, choosing a simple function $$x_n=\sum_{j=1}^{k_n}c_j^{(n)}\chi_{A_j^{(n)}}$$ such that $n\geq{n_0}$ and $\mathcal{A}=\{A_1^{(n)},\dots,A_{k_n}^{(n)}\}$ is a finite collection of pairwise disjoint measurable subsets of $I$ with finite positive measure and assuming that $\mathcal{A}\preceq\mathcal{A'}$ where $\mathcal{A'}$ is a finite collection of such sets we conclude that $T_{\mathcal{A'}}(x_n)=x_n$. In consequence, by \eqref{equ:1:approx:1} and \eqref{equ:2:approx:1} it follows that
	\begin{align*}
	\norm{T_{\mathcal{A'}}(x)-x}{E}{}&\leq\norm{x-x_n}{E}{}+\norm{T_{\mathcal{A'}}(x_n)-x}{E}{}\\
	&=2\norm{x-x_n}{E}{}<{\epsilon}.
	\end{align*}
	Finally, by Proposition 4.3 \cite{Ryan} we get $E$ has the metric approximation property.
\end{proof}

\begin{cor}\label{coro:approx:1}
	Let $E$ be a symmetric space on $I$ and let its dual space $E^*$ be order continuous. Then, for any Banach space $X$ we have $\mathcal{K}(E,X)=E^*\hat{\otimes}_{\epsilon}X$.
\end{cor}

\begin{proof}
	Immediately, since $E^*$ is symmetric, by Theorem \ref{thm:approx:1} it follows that $E^*$ has the metric approximation property. Then, by Corollary 4.13 \cite{Ryan} we get $\mathcal{K}(E,X)$ the space of all compact operators form $E$ to $X$ coincides with the tensor product $E^*\hat{\otimes}_\epsilon{X}$ equipped with the injective norm.
\end{proof}

\begin{cor}\label{coro:approx:2}
	Let $E$ be a symmetric space on $I$. If $E$ is reflexive, then the following assertion are equivalent.
	\begin{itemize}
		\item[$(i)$] $K(E,E)=E\hat{\otimes}_\epsilon{E^*}$ is reflexive.
		\item[$(ii)$] $E\hat{\otimes}_\pi{E^*}$ is reflexive.
		\item[$(iii)$] Every linear bounded operator $T$ from $E$ into $E$ is compact.
	\end{itemize}
\end{cor}

\begin{proof}
	Since $E$ is reflexive symmetric space, by Corollary 4.4 \cite{BS} we get $E$ is order continuous. In consequence, by Theorem \ref{thm:approx:1} it follows that $E$ has the metric approximation property. Next, since $E^*$ is reflexive we conclude that
	\begin{equation*}
	E^*\hat{\otimes}_\epsilon{E^{**}}=E^*\hat{\otimes}_\epsilon{E}=	E\hat{\otimes}_\epsilon{E^{*}},
	\end{equation*}
	whence, by Theorem 4.21 in \cite{Ryan} and by Corollary \ref{coro:approx:1} the end of the proof is completed. 
\end{proof}

The immediate consequence of Theorem 1.e.4 \cite{LinTza} and Theorem \ref{thm:approx:1} is the following result.
\begin{cor}\label{coro:approx:3}
	Let $E$ be a symmetric space. If $E$ is order continuous, then for any Banach space $X$ every compact operator from $X$ into $E$ is approximable. 
\end{cor}

\end{document}